\documentclass[12pt]{amsart}
\usepackage[OT4]{fontenc}
\usepackage{graphics,graphicx}
\usepackage{amsmath}
\usepackage{amssymb,amsfonts}
\usepackage {amsthm}
\usepackage{xcolor}

\usepackage{enumerate}
\usepackage{hyperref}
\usepackage[cp1251]{inputenc}
\usepackage[all]{xy}
\usepackage{cleveref}
\usepackage{lipsum}

\usepackage[active]{srcltx}

\newtheorem{theorem}{Theorem}
\newtheorem{lemma}{Lemma}

\newtheorem{example}{Example}
\newtheorem{remark}{Remark}

\newcommand\myfootnote[1]{
\renewcommand{\thefootnote}{}
\footnotetext{#1}
\def\thefootnote{\@arabic\c@footnote}
}

\makeatletter
\renewcommand{\subsection}{\@startsection{subsection}{2}{0mm}{-\baselineskip}{-5pt}{\it \bf}}
\makeatother

\title{SPECTRA OF LOCALLY MATRIX ALGEBRAS}
\author{OKSANA BEZUSHCHAK}

\begin{document}

\address{Faculty of Mechanics and Mathematics, \\ Taras Shevchenko  National University of Kyiv, \\  60	Volodymyrska Street, Kyiv 01033, Ukraine}

\email{mechmatknubezushchak@gmail.com}

\begin{abstract} We describe spectra of associative (not necessarily unital and not necessarily countable--dimensional) locally matrix algebras. We determine all possible spectra of locally  matrix algebras and give  a new proof of Dixmier--Baranov Theorem. As an application of our description of spectra we determine embeddings of locally matrix algebras.
\end{abstract}

	
\maketitle
\section{Introduction}

Let $\mathbb{F}$ be a ground field. Recall that an associative $\mathbb{F}$--algebra $A$ is called a \emph{locally matrix algebra} (see, \cite{Kurosh}) if for an arbitrary finite subset of $A$ there exists a  subalgebra $B \subset A$ containing this subset and such that  $B$ is isomorphic to a matrix algebra $ M_n(\mathbb{F})$ for some $n\geq 1.$ In what follows we will sometimes identify $B$ and $M_n(\mathbb{F}),$ that is, assume that $M_n(\mathbb{F}) \subset A.$ We call a locally matrix algebra \emph{unital} if it contains~$1.$

Let $A$ be a countable--dimensional unital locally matrix algebra. In \cite{Glimm}, J.G.~Glimm defined  the  Steinitz number $\mathbf{st} (A)$ of the algebra $A$ and proved that $A$ is uniquely determined by $\mathbf{st} (A).$ J.~Dixmier \cite{Diskme}  showed that non-unital countable--dimensional locally matrix algebras over the field of comlex numbers can be parametrized by pairs $(s,\alpha),$ where $s$ is a Steinitz number and $\alpha$ is a nonnegative real number. A.A.~Baranov \cite{Baranov2} extended this parametrization to  locally matrix algebras over arbitrary fields.

In \cite{BezOl}, we defined the  Steinitz number $\mathbf{st} (A)$ for a unital  locally matrix algebra $A$ of an arbitrary dimension. We showed that for a unital  locally matrix algebra $A$ of  dimension $> \aleph_0$ the Steinitz number $\mathbf{st} (A)$ no longer determines $A;$ see, \cite{BezOl_2,Morita_BO}. However,  it determines the universal elementary theory of $A$ \cite{BezOl_2}.

In this paper for an  arbitrary (not necessarily unital and not necessarily countable--dimensional) locally matrix algebra $A,$ we  define a subset of $\mathbb{SN}$ that we call the spectrum of $A$ and denote as $\text{Spec}(A).$ We determine all possible spectra of locally  matrix algebras and give  a new proof of Dixmier--Baranov Theorem. As an application of our description of spectra we determine embeddings of locally matrix algebras.

\section{Spectra of locally  matrix algebras}

Let $ \mathbb{P} $ be the set of all primes and $ \mathbb{N} $ be the set of all positive integers. A
{\it Steinitz   number} (see, \cite{ST})  is an infinite formal
product of the form
\begin{equation*}\label{1}
	\prod_{p\in \mathbb{P}} p^{r_p} \ , \end{equation*}
where   $ r_p\in  \mathbb{N} \cup \{0,\infty\}$ for all $p\in \mathbb{P}.$

Denote by $ \mathbb{SN} $ the set of all Steinitz numbers. Notice, that the set of all positive integers $ \mathbb{N}$ is subset of $\mathbb{SN}$.    The numbers  $ \mathbb{SN}\setminus   \mathbb{N}  $  are called {\it infinite}  Steinitz numbers.

Let $A$ be a  locally matrix algebra with a unit $1$ over a field $F$ and let $D(A)$ be  the  set of all positive integers $n$ such that  there is a subalgebra  $A'$, $1 \in A'\subseteq A$, $A' \cong  M_n(F)$. Then the  least common multiple of the set $D(A)$ is called the {\it Steinitz  number  of the algebra} $A$ and denoted as $\mathbf{st}(A)$; see, \cite{BezOl}.

Now let $A$ be a (not necessarily unital) locally matrix algebra. For an arbitrary  idempotent $0 \neq e \in A$ the subalgebra $eAe$ is a unital locally matrix algebra. That is why we can talk about its Steinitz  number  $\mathbf{st}  (eAe).$ The subset
 \begin{equation*}\label{Spec(A)} \text{Spec}  (A) \ = \ \big\{ \ \mathbf{st}  (eAe) \ | \ e \ \in \ A, \ e \ \neq 0, \ e^2 \ = \ e \  \big\},
 \end{equation*} where  $e$ runs through all nonzero idempotents of the algebra  $A,$ is called the \emph{spectrum} of the algebra $A.$

For a Steinitz  number $s$ let $ \Omega  (s)$ denote the set of all natural numbers $n  \in  \mathbb{N}$ that  divide $s$.

 For a Steinitz  numbers $s_1,$ $s_2$ we say that $s_1$ \emph{finitely divides}  $s_2$ if there exists $b  \in \Omega  (s_2)$ such that $s_1  = s_2 / b$ (we denote: $ s_1 \, \big|\,_{fin} \,s_2$).

Steinitz  numbers $s_1,$ $s_2$ are \emph{rationally connected} if $s_2=q\cdot s_1,$ where $q$ is some rational number.

We call a subset  $S \subset \mathbb{SN}$ \emph{saturated} if
 \begin{enumerate}
\item[$1)$] any two Steinitz  numbers from $S$ are rationally connected;
\item[$2)$] if $s_2 \in S$ and $ s_1 \, \big|\,_{fin} \,s_2 $  then $s_1 \in S;$
\item[$3)$] if $s,$ $ns \in S,$  where $n\in \mathbb{N},$ then  $is \in S$ for any $i,$ $1 \leq i \leq n.$
\end{enumerate}

\begin{theorem}\label{Th1_Spec} For an arbitrary locally matrix algebra $A$ its spectrum is a saturated subset of $\mathbb{SN}.$
\end{theorem}

Let us consider examples of saturated subsets of $\mathbb{SN}.$

\begin{example}
For an arbitrary natural number $n$ the set  $\{1,2,\ldots,n\}$ is saturated.
\end{example}

\begin{example}
Let $s$ be a Steinitz  number. The set  $$S (\infty,s )  \ :=  \ \Big\{ \ \frac{a}{b} \ \cdot \ s \ \Big| \ a \ \in \ \mathbb{N}, \ b \ \in \ \Omega (s) \ \Big\}$$ is saturated. For an arbitrary Steinitz number $s' \in S (\infty,s )$ we have  $S (\infty,s ) =S  (\infty,s' ).$ If $s\in \mathbb{N}$ then $S (\infty,s )=\mathbb{N}.$
\end{example}

\begin{example} Let $r$ be a real number, $1 \leq  r < \infty .$ Let  $s$ be an infinite Steinitz  number. The set $$ S  (r,s) \ = \ \Big\{ \ \frac{a}{b} \ s \ \Big| \ a, \ b \ \in \ \mathbb{N}; \ b \ \in \ \Omega  (s), \ a \ \leq \ r   b \ \Big\}$$  is saturated.\end{example}

\begin{example} Let  $s$ be an infinite Steinitz  number and let  $r= u / v$ be a rational number;  $u, v\in \mathbb{N},$   $v\in \Omega \ (s).$ Then the set
$$ S^{+}  (r,s)  =  \Big\{  \frac{a}{b}  s \ \Big| \ a,  b  \in  \mathbb{N}; \ b  \in  \Omega  (s), \ a  <  r  b  \Big\}$$ is saturated. \end{example}

\begin{theorem}\label{Th_II.1.11_new} Every saturated subset of $\mathbb{SN}$ is one of the following sets:
\begin{enumerate}
  \item[$1)$] $\{ 1,2,\ldots, n\},$  $n\in \mathbb{N},$ or $\mathbb{N};$
  \item[$2)$] $S  (\infty, s),$ $s\in \mathbb{SN} \ \diagdown \ \mathbb{N};$
  \item[$3)$] $S  (r,s),$ where $s\in \mathbb{SN} \ \diagdown \ \mathbb{N},$ $r\in [1,\infty);$
  \item[$4)$] $S^{+}  (r,s),$ where $s\in \mathbb{SN} \ \diagdown \ \mathbb{N},$ $r=u/v,$  $u\in \mathbb{N},$  $v\in \Omega  (s).$
\end{enumerate}
\end{theorem}

\begin{remark} The real number $r$ above is the inverse of the density invariant of Dixmier--Baranov.
\end{remark}

\begin{theorem}\label{Th_II.2.1_new} \begin{enumerate}
                                    \item[$(1)$] For any saturated subset $S\subseteq \mathbb{SN}$ there exists a countable--dimensional locally matrix algebra $A$ such that $\emph{\text{Spec}}  (A)=S.$
                                    \item[$(2)$] If $A,$ $B$ are countable--dimensional locally matrix algebras and  $\emph{\text{Spec}}  (A)  =  \emph{\text{Spec}}  (B)$ then $ A  \cong  B.$
                                  \end{enumerate}
\end{theorem}

\begin{remark} The part $(2)$ of Theorem \emph{\ref{Th_II.2.1_new}} is a new proof of Dixmier--Baranov Theorem.
\end{remark}

Which spectra above correspond to unital algebras?

\begin{theorem}\label{Th4_Spec} A locally matrix algebra $A$ is unital if and only if  $\emph{\text{Spec}}  (A) =\{1,2,\ldots,n\},$ where $n\in \mathbb{N},$ or $\emph{\text{Spec}}  (A) = S  (r,s),$ where  $s\in \mathbb{SN} \ \diagdown \ \mathbb{N},$  $r=u/v,$  $u,v\in \mathbb{N},$   $v\in \Omega  (s).$
\end{theorem}

\begin{proof}[Proof of Theorem $\ref{Th1_Spec}$] In what follows, we assume that $A$ is a locally matrix $\mathbb{F}$--algebra.  Recall the partial order on the set of all idempotents of $A:$ for idempotents $e,$ $f \in A$ we define  $e  \geq  f $  if $f   \in  e\,A\,e.$

We claim that for arbitrary idempotents $e_1,$ $e_2 \in A$ there exists an idempotent  $e_3 \in A$ such that  $e_1\leq e_3,$ $e_2\leq e_3.$ Indeed, there exists a subalgebra  $A'\subset A$ such that  $e_1, $ $ e_2  \in  A'$  and $ A'\cong M_n(\mathbb{F}),$ $ n\geq 1.$ Let  $e_3$ be the identity element of the subalgebra $A'.$ Then $ e_1  \leq  e_3,$ $ e_2  \leq  e_3.$

Now suppose that the locally matrix algebra $A$ is unital. Let $a\in A.$ Choose a subalgebra $A'\subset A$ such that $1\in A',$ $a\in A'$ and $A' \cong M_n(\mathbb{F}),$ $n\geq 1.$ Let $r$ be the range of the matrix $a$ in $A'.$ Let $$ r(a) \ = \ \frac{r}{n}, \quad 0\leq r(a) \leq 1.$$ V.M.~Kurochkin \cite{Kurochkin} noticed that the number $r(a)$ does not depend on a choice of a subalgebra $A'.$ We call $r(a)$ the \emph{relative range}  of the element $a.$ In \cite{Morita_BO}, we showed  that if $A$ is a unital locally matrix algebra and $e\in A$ is an idempotent, then $\mathbf{st}(eAe)=r(e)\cdot \mathbf{st}(A).$

Now let $A$ be a not necessarily unital locally  matrix algebra. Let $e_1,$ $e_2\in A$ be idempotents. Choose an idempotent $e_3\in A$ such that $e_1 \leq e_3,$ $e_2 \leq e_3,$ i.e. $e_1, $ $e_2\in e_3 A e_3.$ Let  $q_1,$ $q_2$ be relative ranges of the idempotents $e_1,$ $e_2$ in unital locally  matrix algebra $e_3 A e_3.$ Then  $$\mathbf{st}  (e_1\, A \, e_1) \ = \ q_1 \  \mathbf{st}  (e_3\, A \, e_3),  \ \ \ \ \mathbf{st}  (e_2\, A \, e_2) \ = \ q_2 \  \mathbf{st}  (e_3\, A \, e_3). $$ This implies that the Steinitz  numbers $\mathbf{st} (e_1\, A \, e_1),$ $\mathbf{st}  (e_2\, A \, e_2)$ are rationally connected. We have checked the condition 1) from the definition of saturated sets.

Let $0 \neq e\in A$ be an idempotent. Let $s_2  =  \mathbf{st}  (e A e),$ $k \in  \Omega  (s_2)$ and let $s_1  = s_2 / k.$ The unital locally matrix algebra  $eAe$ contains a subalgebra $ e\in M_k(\mathbb{F})\subset eAe.$  Consider the matrix unit  $e_{11}$ of the algebra $M_k(\mathbb{F}).$ The relative range of the idempotent $e_{11}$ in the unital algebra $eAe$ is equal to $1/ k.$ Hence $$\mathbf{st}  (e_{11}\, A\, e_{11}) \ = \ \frac{1}{k} \  \mathbf{st} (e\, A\, e) \ = \ s_1, \quad s_1 \in \text{Spec} (A). $$ We have checked the condition 2).

Now let $n\geq 1.$ Suppose that Steinitz numbers $s$ and $ns$ lie in $\text{Spec}(A).$ It means that there exist idempotents  $e_1,$ $e_2 \in A$ such that $ s  =   \mathbf{st}  (e_1 A  e_1),$ $ n   s =   \mathbf{st}  (e_2 A  e_2).$ There exists a matrix subalgebra $M_k(\mathbb{F})  \subset  A$ that contains $e_1$ and $e_2.$ As above, let $e_3$ be the identity element of the algebra  $M_k(\mathbb{F}).$ Let  $\text{rk}(e_i)$ be the range of the idempotent $e_i$ in the  matrix algebra $M_k(\mathbb{F}).$ We have $$ s  =   \frac{\text{rk}(e_1)}{k}  \cdot \mathbf{st} (e_3\, A \, e_3), \ \ \ \ n\ s \ = \ \frac{\text{rk}(e_2)}{k}  \cdot  \mathbf{st}  (e_3\, A \, e_3),$$ which implies $\text{rk}(e_2)  =  n  \cdot   \text{rk}(e_1).$ In particular, $ n  \cdot   \text{rk}(e_1)  \leq  k.$ Let $1 \leq  i \leq n .$ Consider the idempotent $$ e  =  \text{diag} \ \big( \ \underbrace{1, \ 1, \ \ldots, \ 1,}_{i\cdot \text{rk}(e_1)} \ 0, \ 0, \ \ldots, \ 0\ \big) $$ in the matrix algebra $M_k(\mathbb{F}).$ We have $$\mathbf{st} (e\, A \, e)  =    \frac{i \cdot \text{rk}(e_1)}{k}  \cdot  \mathbf{st}   (e_3\, A \, e_3)  = i \ \cdot \ \mathbf{st}   (e_1\, A \, e_1) =  i \ s.$$ We showed that $is\in \text{Spec}  (A).$ Hence $\text{Spec}  (A)$ is a saturated subset of $\mathbb{SN}.$ It completes the proof of Theorem \ref{Th1_Spec}. \end{proof}

\section{Classification of saturated subsets of $\mathbb{SN}$}

Our aim in this section is to classify all saturated subsets of $\mathbb{SN}.$ We remark that if at least one Steinitz  number from a saturated set $S$ is infinite then by the condition $1)$ all Steinitz  numbers from $S$ are infinite.

Let $S$ be a saturated subset of $\mathbb{SN}.$ For a Steinitz  number $s\in S$ and for a natural  number  $b\in \Omega  (s)$ let $$ r_s(b) \ = \ \max \ \Big\{ \ i \ \geq \ 1 \ \Big| \ i \ \cdot \ \frac{s}{b} \ \in \ S \ \Big\}.$$

\begin{lemma}\label{Lem_II.1.2_new}   If there exists a Steinitz  number $s_0 \in S$ and a natural  number $b_0 \in \Omega \ (s_0)$ such that
$ r_{s_0}(b_0) \ = \ \infty$ then for any $s\in S$ and any $b\in \Omega \ (s)$  we have
$ r_s(b)= \infty.$ \end{lemma}

\begin{proof} Let us show at first that  $r_{s_0}(b)  = \infty $ for any $ b  \in  \Omega  (s_0).$ Indeed, there exists a natural  number $c  \in  \Omega  (s_0)$ such that both $b_0$ and $b$ divide $c.$  Then for an arbitrary  $i\geq 1$ we have $$i \ \cdot \ \frac{s_0}{b_0} \ = \ \Big( \ i \  \cdot \ \frac{c}{b_0} \ \Big) \ \cdot \ \frac{s_0}{c} \ \in \ S. $$ This implies that $ r_{s_0}(c)  =  \infty.$ Hence, $$i \ \cdot \ \frac{s_0}{b} \ = \ \Big( \ i \  \cdot \ \frac{c}{b} \ \Big) \ \cdot \ \frac{s_0}{c} \ \in \ S, $$ which proves the claim.

Now choose an  arbitrary Steinitz  number $s\in S.$ By the condition $1),$ the Steinitz  numbers $s$ and $s_0$ are rationally connected, i.e. there exist $a\in \mathbb{N},$ $ b \in  \Omega  (s_0)$ such that $ s  = ( a / b ) \cdot  s_0.$ By condition $2),$  $ s_0 /b   \in  S. $ Choose a natural  number $c  \in  \Omega  (s_0 / b). $ Then $ c  \in   \Omega  (s)$  and $ bc  \in  \Omega  (s_0). $ For an arbitrary $i\geq 1$  we have $i  \cdot s / c  =  i  \cdot  a  \cdot s_0 / (bc)   \in  S $ since $r_{s_0}(bc)  =  \infty . $ This implies  $ r_s(c)   =  \infty $ and completes the proof of the lemma. \end{proof}

If a saturated set satisfies the assumptions of Lemma \ref{Lem_II.1.2_new} then it is referred to as  a set of \emph{infinite type}. Otherwise, we talk about a saturated set of \emph{finite type}.

\begin{lemma}\label{Lem_II.1.3_new} \begin{enumerate}
                                       \item[$1)$] For an arbitrary Steinitz  number $s_0\in \mathbb{SN}$ the set $$S  (\infty,s_0 )  \ :=  \ \Big\{ \ \frac{a}{b} \ \cdot \ s_0 \ \Big| \ a \ \in \ \mathbb{N}, \ b \ \in \ \Omega  (s_0) \ \Big\}$$ is a saturated set of infinite type.
                                       \item[$2)$] If $S$ is a  saturated set of  infinite type, then for an arbitrary Steinitz  number $s\in S$ we have $S=S  (\infty,s ) .$                                      \end{enumerate}
\end{lemma}

\begin{proof} We have to show that the set  $S (\infty,s_0 ) $ satisfies the conditions $1),$ $2),$ $3).$ The condition $1)$ is obvious. Let $ s  = ( a / b) \cdot  s_0,$ $b  \in  \Omega  (s_0).$ Without loss of generality, we assume that  $a$ and $b$ are coprime. Let $c\in \Omega (s)$ and let  $ d  =  \text{gcd} (c,a)$ be the greatest common divisor of $a$ and $c, $ $a = a\,'  d, $ $ c  =  c\,'  d,$ the numbers $a\,',$ $c\,'$ are coprime. Then $ a \cdot s_0 / (bc)  =  a\,' \cdot s_0 / (bc\,'), $ which implies that $d  c\,'  \in  \Omega  (s_0).$ Hence $$ \frac{s}{c}  \ = \ \frac{a}{bc} \ \cdot \ s_0 \ = \  \frac{a\,'}{bc\,'} \ \cdot \ s_0 \ \in \ S    (\infty,s_0 ) .$$ We have checked the condition $2).$

Let us check the  condition $3).$ Choose $ s  = ( a / b)  \cdot  s_0  \in  S  (\infty,s_0 ),$ $ b  \in  \Omega  (s_0). $ Let $c \in \Omega   (s).$ We need to check that for any $i \geq 1 $ $$ i  \ \cdot \ \frac{s}{c} \ = \  \frac{ia}{bc} \  \cdot \ s_0 \ \in \ S   (\infty,s_0 ). $$ Let $ a / (bc )  =  a\,'  / b\,', $ where the natural numbers $ a\,',$ $  b\,' $ are coprime. Since  $$ \frac{a}{bc} \ \cdot \ s_0 \ = \ \frac{s}{c}\in \mathbb{SN}$$ it follows that  $b\,' \in \Omega   (s_0).$ Hence, $ i  \cdot (a\,' / b\,')  \cdot  s_0  \in   S   (\infty,s_0 ),$ which implies that $S  (\infty,s_0 )$ satisfies the condition $3)$ and therefore is saturated.

Let $S$ be a  saturated subset of $\mathbb{SN}$ of infinite type. Choose  $s_0 \in S.$ Our aim is to show that  $ S  =  S  (\infty,s_0 ). $ Since the subset $S$ is of infinite type it follows that  $r_s(b)=\infty$ for any $s\in S,$ $b\in \Omega(s).$ In particular, $$ S   (\infty,s_0 ) \ = \ \Big\{ \ \frac{a}{b} \ \cdot \ s_0 \ \Big| \ s \ \in \Omega   (s_0) \ \Big\}  \ \subseteq \ S.$$ An arbitrary Steinitz  number $s\in S$ is rationally connected  to $s_0,$ hence there exist   $a,$ $b \in \mathbb{N}$ such that  $ s  =   (a / b ) \cdot  s_0 .$ Without loss of generality, we assume that  $a$ and $b$ are coprime, which implies $b \in \Omega  (s_0).$ We proved that  $s \in S  (\infty,s_0 ).$ \end{proof}

Now let  $S\subset \mathbb{SN}$ be a saturated subset of finite type, that is, for any   $s \in S,$ $d \in \Omega \ (s)$  we have $$r_s(b) \ = \ \max \ \Big\{ \ i \ \in \ \mathbb{N} \ \big| \ i \ \cdot \ \frac{s}{b} \ \in \ S \ \Big\} \ < \ \infty.$$ By the condition $3),$  $$\Big\{ \ i \ \in \ \mathbb{N} \ \big| \ i \ \cdot \ \frac{s}{b} \ \in \ S \ \Big\} \ = \ \big[ \ 1 \ , \ r_s(b) \ \big]. $$ Since $ b  \cdot  (s / b ) \in  S $  it follows that $ b  \leq  r_s(b).$ Choose a Steinitz  number $ s  \in  S$ and  two natural numbers $ b,$ $ c   \in  \Omega  (s)$ such that $ b$ divides $c.$ If $ i  \cdot  (s / b )  \in  S $ then $( ic / b)  \cdot  ( s/ c)  \in  S.$ Hence $ r_s(b)  \cdot  ( c / b )  \leq  r_s(c).$ In other words,
\begin{equation}\label{Big_frac_1} \frac{r_s(b)}{b} \ \leq \ \frac{r_s(c)}{c}.
\end{equation}
Let  $i \in \mathbb{N},$ $s / c  \in  S$ and let $k$ be a maximal nonnegative integer such that  $  k  \cdot ( c / b ) \leq  i.$ By the condition $3),$ $ k  \cdot  (c / b)  \cdot (s / c ) \in  S, $ hence $ k  \cdot  (s / b)   \in  S.$  So, $ k  \leq r_s(b).$ We proved that
\begin{equation}\label{Big_frac_2}
\Big[ \ \frac{r_s(c)}{c \, / \, b} \ \Big] \ \leq \ r_s(b) .
\end{equation}
The inequalities  (\ref{Big_frac_1}), (\ref{Big_frac_2})  imply $$\Big[ \ \frac{r_s(c)}{c \, / \, b} \ \Big] \ \leq \ r_s(b) \ \leq \  \frac{r_s(c)}{c \, / \, b} .$$  Hence
\begin{equation}\label{Big_frac_3} r_s(b) \ = \ \Big[ \ \frac{r_s(c)}{c \, / \, b} \ \Big]  .\end{equation}
In particular, $$ \frac{r_s(c)}{c \, / \, b} \ - \ 1 \ < \ r_s(b), \quad  \frac{r_s(c)}{c \, / \, b} \ < \ r_s(b) \ + \ 1.$$ Dividing by $b,$ we get
 \begin{equation}\label{Big_frac_4} \frac{r_s(b)}{b} \ \leq \ \frac{r_s(c)}{c} \ < \ \frac{r_s(b)}{b} \ + \ \frac{1}{b}.
 \end{equation}

\begin{lemma}\label{Lem_II.1.4_new}  Let $ S \subset \mathbb{SN}$ be a saturated subset of finite type and let $s\in S$ be an infinite Steinitz  number. Then there exists a limit $$r_S(s) \ = \ \lim_{\substack{ \\ b \ \in \ \Omega  (s)\\b \,\rightarrow \,\infty}}  \ \frac{r_s(b)}{b} \ ,  \quad 1 \ \leq \ r_S(s) \ < \infty.$$
\end{lemma}
If the set $S$ is fixed then we denote  $r_S(s)= r(s).$

\begin{remark} The limit $r(s)$ is equal to the inverse of the density invariant of Dixmier--Baranov \emph{(\cite{Baranov2,Diskme})}.
\end{remark}

\begin{proof}[The proof of Lemma $\ref{Lem_II.1.4_new}$] The set $ \{ r_s(b) / b \ | \ b  \in  \Omega  (s) \} $ is  bounded from above. Indeed, choose $b_0 \in \Omega  (s).$ For an arbitrary  $b \in \Omega  (s)$ there exists $c \in \Omega  (s)$ that is a common multiple for  $b_0$ and $b.$ Then by (\ref{Big_frac_1}) and (\ref{Big_frac_4}), $$\frac{r(b)}{b} \ \leq \ \frac{r(c)}{c} \ < \ \frac{r(b_0)}{b_0} \ + \ \frac{1}{b_0} .$$

Let   $$ r \ = \ r(s) \ = \ \sup \ \Big\{ \ \frac{r_s(b)}{b} \ \Big| \ b \ \in \ \Omega  (s) \ \Big\}.$$ Clearly, $ 1\leq r < \infty.$ Choose $\varepsilon > 0.$ Let $ N(\varepsilon)  = [  2r/ \varepsilon ]  +  1.$  We will show that for any  $b\in \Omega \ (s),$ $b\geq N(\varepsilon),$ we have $ r  -  \varepsilon  <   r_s(b) / b. $

Indeed, let $ b  \in  \Omega  (s),$ $ b  \geq N(\varepsilon)  >  2r / \varepsilon. $ Then $ 1/ b  < \varepsilon / (2r) \leq \varepsilon / 2 . $ There exists a natural number $b_0 \in \Omega \ (s)$ such that $ r  -  \varepsilon / 2  < r_s(b_0) / b_0.$  Let $c\in \Omega  (s)$ be a common multiple of  $b_0$ and $b.$ Then  (\ref{Big_frac_4}) implies  $$ \frac{r(b)}{b} \ > \ \frac{r(c)}{c} \ - \ \frac{1}{b} \ \geq \ \frac{r(b_0)}{b_0} \ - \ \frac{1}{b} \ > \ r \ - \ \frac{\varepsilon}{2} \ - \  \frac{\varepsilon}{2} \ = \ r \ - \ \varepsilon.$$ So,  $$r \ = \ \lim_{\substack{ \\ b \ \in \ \Omega  (s)\\b \,\rightarrow \,\infty}}  \ \frac{r_s(b)}{b}  $$ and this completes the proof of the lemma. \end{proof}

\begin{lemma}\label{Lem_II.1.5_new}  Let $s, $ $ s\,' \in S$ be infinite  Steinitz  numbers, $ s\,'  =  (a / b) \cdot s;$ $ a,$ $b  \in  \mathbb{N} ;$ $ b \in  \Omega  (s).$ Then  $r(s\,')  =  (a/ b) \cdot r(s).$
\end{lemma}

\begin{proof} It is sufficient to show that if  $s ,$ $ ms  \in S;$ $m \in \mathbb{N},$ then $ m  \cdot  r(m  s)  =   r(s). $

Suppose that $b \in \Omega  (s)$  and $i  \cdot (m  s/ b) \in  S.$ Then $ i  \cdot  m  \cdot  (s/b) \in  S. $ Hence $ r_{ms}(b)  \cdot  m  \leq  r_s(b)$ and, therefore, $r(m  s)  \cdot  m  \leq  r(s).$

On the other hand, if $ i  \cdot (s/b) \in  S $ then $ [ i / m ]   \cdot  m \leq  i$ and, therefore, $ [ i / m ]   \cdot  m  \cdot  (s / b)  \in  S .$ We showed that $$\Big[ \ \frac{r_s(b)}{m} \ \Big] \ \leq \ r_{ms}(b), \ \ \ \ \frac{r_s(b)}{m} \ - \ 1 \ < \ r_{ms}(b), $$ $$  \frac{1}{m} \ \cdot \ \frac{r_s(b)}{b} \ - \ \frac{1}{b} \ < \ \frac{r_{ms}(b)}{b} .$$ Assuming  $b\rightarrow \infty$ we get  $ (1/ m)  \cdot  r(s)  \leq  r(m s),$ which completes the proof of the lemma. \end{proof}

 In the inequality  (\ref{Big_frac_4}), let $c \rightarrow \infty.$ Then $$\frac{r_s(b)}{b} \ \leq \ r(s) \ \leq \ \frac{r_s(b)}{b} \ + \ \frac{1}{b}, \quad r_s(b) \ \leq \ r(s) \ b \ \leq \ r_s(b) \ + \ 1.$$ If the number $r(s)$ is irrational then  $r_s(b)  =  [  r(s)  b  ] $ for all $ b  \in  \Omega  (s).$

Now suppose that the number $r=r_s(b)$ is rational;  $r=u/v;$  $u,$ $v$ are coprime. If a number  $b\in \Omega(s)$ is not a multiple of $v$ then, as above, $r_s(b)  = [ (u / v) \cdot b ].$ If  $b$ is  a multiple of $v$ then $$ r_s(b) \ = \ \left[
                                                       \begin{array}{ll}
                                                         r \ b & \hbox{or} \\
                                                         r \ b \ - \ 1 & \hbox{.}
                                                       \end{array}
                                                     \right.
 $$

\begin{lemma}\label{Lem_II.1.6_new} If at least for  one number   $b_0 \in  \Omega  (s)  \bigcap  v  \mathbb{N}$ we have  $r_s(b_0) = r  b_0$ then  for all  $b  \in  \Omega  (s)  \bigcap  v  \mathbb{N} $ we have $ r_s(b)=r  b.$ \end{lemma}

\begin{proof} Let $b, $ $ c   \in  \Omega  (s)  \bigcap  v  \mathbb{N}$ and $b$ divides $ c.$ If $r_s(b) = rb$ then by the inequality  (\ref{Big_frac_1}),  we have  $$ r \ =  \ \frac{r_s(b)}{b} \ \leq \ \frac{r_s(c)}{c},$$ which implies $r_s(c)=r  c.$ On the other hand, if $r_s(c)=r  c$ then by the inequality  (\ref{Big_frac_4}),  $$ r \ =  \ \frac{r_s(c)}{c} \ < \ \frac{r_s(b)}{b} \ + \ \frac{1}{b},$$ which implies $r_s(b)> r b -1.$ Hence  $r_s(b)= r b.$ We showed that $r_s(b)=rb$ if and only if  $r_s(c)=rc.$

Now choose $b_1,$ $ b_2   \in  \Omega  (s)  \bigcap  v  \mathbb{N}$ and suppose that $ r_s(b_1)  =  r  b_1.$  There exists  $c   \in  \Omega  (s)  \bigcap  v  \mathbb{N}$ such that  both $ b_1$ and $b_2$ divide $ c.$  In view of the above, $r_s(b_1) = r b_1$ implies  $r_s(c) = r c$ which implies $r_s(b_2) = r b_2.$ This completes the proof of the lemma.  \end{proof}

Recall that for an infinite Steinitz  number $s$ and a real number  $r,$ $1 \leq  r < \infty, $ $$ S  (r,s) \ = \ \Big\{ \ \frac{a}{b} \ s \ \Big| \ a, \ b \ \in \ \mathbb{N}; \ b \ \in \ \Omega  (s), \ a \ \leq \ r   b \ \Big\},$$  $$ S^{+}  (r,s) \ = \ \Big\{ \ \frac{a}{b} \ s \ \Big| \ a, \ b \ \in \ \mathbb{N}; \ b \ \in \ \Omega  (s), \ a \ < \ r  b \ \Big\}.$$ If $r$ is an irrational number or $r=u / v,$ the integers $u,$ $v$ are coprime and  $v\not\in \Omega \ (s)$ then $S  (r,s)  =  S^{+}  (r,s).$ If   $r= u / v,$  $v\in \Omega \ (s)$ then $S^{+}  (r,s)  \subsetneqq  S  (r,s).$

\begin{lemma}\label{Lem_II.1.7_new} The subsets $S  (r,s)$ and  $ S^{+}  (r,s)$ are saturated.
\end{lemma}

\begin{proof} The condition $1)$  in the definition of saturated subsets is obviously satisfied. Let us  check the condition $2).$ Let $ (a/ b) \cdot s  \in  S  (r,s) $ (respectively, $(a/b) \cdot s  \in  S^{+}  (r,s) ),$ where $a,$ $b$ are coprime natural  numbers, $b \in \Omega  (s).$ Then $ a  \leq  r  b $ (respectively $a  <    r  b $). Suppose that $ c \in  \Omega  (  \frac{a}{b}  s ).$ We need to show that  $ (a\cdot s) / (b \cdot c)  \in  S  (r,s)$ (respectively, $ (a \cdot s ) / (b \cdot c) \in  S^{+}  (r,s)).$ Let $ d  =  \text{gcd} (a,c),$ $ a  =  d a\,',$ $ c  =  d  c\,'. $ Then $$\frac{a \ s}{b \ c} \ = \ \frac{a\,'}{b \ c\,'} \ s \ \in \ \mathbb{SN}. $$ Since the number $b c\,'$ is coprime with $a\,'$ it follows that $b c\,' \in \Omega  (s).$ The inequality  $a\,'  \leq  r  b c\,' $ (respectively $a\,'  <    r  b c\,'$) is equivalent to the inequality $a  \leq  r  b c$  (respectively $ a  <    r  b c$). The latter  inequality follows from $a  \leq  r  b$  (respectively $ a  <    r  b $). The  condition $2)$ is verified.

Let us check the condition $3).$ As above, we  assume that  $a,$ $b$ are coprime natural numbers, $b\in \Omega (s)$ and $a/ b  \in  S  (r,s) $ (respectively $a / b  \in  S^{+}  (r,s) $). Let $c  \in  \Omega ( (a/ b ) \cdot s ),$ $\text{gcd} (a,c) =  d,$ $ a  =  d a\,',$ $ c  =   dc\,' .$ We have shown above that $bc\,' \in \Omega  (s).$ Let $n\in \mathbb{N}$ and $n   \cdot (as / (bc) )  \in  S  (r,s) $ (respectively $ n \cdot (as / (bc) ) \in S^{+}  (r,s)$). Then $n  a\,'  \leq  rbc\,'$ (respectively $ n  a\,'  <  rbc\,' $).  This immediately implies that for any  $i,$ $1\leq i \leq n,$ we have $i  a\,'  \leq  rbc\,' $ (respectively $ i a\,'  <  rbc\,'$). Hence, $ i  \cdot (as /b)   \in  S  (r,s)$ (respectively $ i  \cdot (as/ b)   \in  S^{+} (r,s)$).  \end{proof}

\begin{lemma}\label{Lem_II.1.8_new}  Let $r=u/ v,$ where $u,$ $v$ are coprime natural  numbers. Let $s$ be an infinite  Steinitz  number and   $v\in \Omega \ (s).$ Then the set $ S^{+}  (r,s)$ is not equal to any of the sets $S  (r\,',s\,'),$ $r\,'\in [1,\infty),$ $s\,' \in \mathbb{SN}.$ \end{lemma}

\begin{proof} Let $s_2 \in S(r,s_1)$ (respectively $s_2 \in S^{+}(r,s_1)$).  Then $ s_2  =  (a/ b) \cdot  s_1,$ where $a,$ $ b  \in \mathbb{N},$ $b\in \Omega(s_1).$ By Lemma \ref{Lem_II.1.5_new}, $$ S   (r,s_1) \ = \ S    \Big(  r \ \frac{b}{a}, s_2  \Big)   \ \  \text{\Big(respectively} \   \  S^{+}   (r,s_1) \ = \ S^{+}    \Big( r \ \frac{b}{a}, s_2  \Big) \Big).$$ We showed that the set  $S  (r,s)$ (respectively $S^{+}  (r,s)$) is determined by any Steinitz  number $s\,' \in S  (r,s)$ (respectively $s\,' \in S^{+}  (r,s)$) with an appropriate recalibration of $r.$

Let $ S  =  S  (r_1, s_1)  =  S^{+}  (r_2, s_2). $ Choosing an arbitrary Steinitz  number  $s\in S$ we get $ S  (r\,'_1, s)  =  S^{+}  (r\,'_2, s)$ for  some  $r\,'_1,$ $r\,'_2 \in [1,\infty).$ The number $r\,'_2 = u/v$ is rational, $\text{gcd}  (u,v)   =  1$ and $ v  \in  \Omega  (s). $

The number $r$ is uniquely determined by a saturated subset  $S$ and a choice of  $s\in S.$ Hence $r\,'_1  =  r\,'_2. $ Now it remains to notice that for a rational number  $r=u/v,$ $\text{gcd}  (u,v)   =  1$ and  an infinite Steinitz  number $s,$ such that  $v\in \Omega  (s),$ we have  $S  (r,s)  \neq  S^{+}  (r,s). $ This completes the proof of the lemma. \end{proof}

\begin{lemma}\label{Lem_II.1.9_new} Let $S   \subset  \mathbb{SN}   \diagdown  \mathbb{N}$ be a saturated subset of finite type,  $s   \in  S,$ $r =  r_S(s)  \in   [1,  \infty).$ Then  $S   =  S  (r,s) $  or $  S  =  S^{+}  (r,s).$ \end{lemma}

\begin{proof} Recall that for a natural number $b \in \Omega  (s)$ we defined $$r_s (b) \ = \ \max \ \Big\{ \ i \ \in \ \mathbb{N} \  \ \Big| \ \ i \ \frac{s}{b}  \  \in \ S \ \Big\}.$$ We showed that if $r$ is an irrational number or  $r=u/v;$ $u,$ $v$ are coprime  and  $v\not\in \Omega \ (s)$ then $r_s(b)  =  [ rb  ] $ for an arbitrary $ b  \in  \Omega  (s).$

An arbitrary Steinitz  number $s\,'\in S$ is representable as  $ s\,' =  (a /b) \cdot  s,$ where $a,$ $b$ are coprime natural numbers.  Clearly,  $ b\in \Omega  (s)$ and $a  \leq  r_s(b)  =  [ rb ].$ That is why in the case when  $r$ is irrational or  $r=u/v,$ $\text{gcd}  (u,v)   =  1,$  $v \not\in \Omega  (s)$ we have $ S  =  S   (r,s)  =  S^{+}  (r,s).$

Suppose now that $  r  =  u/ v,$ $\text{gcd}  (u,v)  =  1,$ $  v  \in  \Omega  (s).$ If $b  \in   \Omega  (s)  \diagdown  v\mathbb{N}$ then  as above $r_s(b)  =  [  rb ].$ By Lemma \ref{Lem_II.1.6_new}, either for all $b  \in   \Omega  (s)  \cap  v\mathbb{N}$ we have $ r_s(b)  =  rb $ or for all  $b  \in   \Omega  (s)  \cap  v\mathbb{N} $ we have $ r_s(b)  =  rb  -  1. $
In the first case  $S = S (r,s),$  in the second case  $S = S^{+}  (r,s).$  \end{proof}

\begin{lemma}\label{Lem_II.1.10_new} Let $S\subseteq \mathbb{N}$ be a saturated subset. Then either  $S=\{1,2,\ldots,n\}$ for some $n \in \mathbb{N}$ or $S=\mathbb{N}.$ \end{lemma}

\begin{proof} First, notice that the subsets $\{1,2,\ldots,n\}$ and $\mathbb{N}$ are saturated.

Now let   $S\subseteq \mathbb{N}$ be a saturated subset.  If $n\in S$ then $n\in \Omega  (n)$ and $n  \cdot (n/n)  \in  S. $ By the condition $3),$ all natural numbers $ i  =  i  \cdot  (n/ n),$ $ 1\leq i \leq n,$ lie in $S.$ This implies the assertion of the lemma \end{proof}

Now, Theorem \ref{Th_II.1.11_new} follows from Lemmas \ref{Lem_II.1.9_new}, \ref{Lem_II.1.10_new}.

\section{Countable--dimensional locally matrix algebras}

For an algebra $A$ and an idempotent $0\neq e\in A$ we call the subalgebra $eAe$ a \emph{corner} of the algebra $A.$

Let $ A_1  \subset  A_2  \subset  \cdots$ be an ascending chain of unital locally matrix algebras,  $A_i$ is a corner of the algebra $A_{i+1},$ $i\geq 1,$ $$ A \ = \ \bigcup_{i=1}^{\infty} \ A_i .$$ Clearly, $\text{Spec}  (A_1)  \subseteq \text{Spec} (A_2)  \subseteq  \cdots.$

\begin{lemma}\label{Lem_II.2.2_new}  $$\emph{Spec}  (A) \ = \ \bigcup_{i=1}^{\infty} \ \emph{Spec}  (A_i).$$ \end{lemma}

\begin{proof} For an arbitrary idempotent $e\in A_i$ we have $e A_i e   =   e A e,$ hence $\text{Spec}  (A_i) \subseteq  \text{Spec}  (A) .$ On the other hand, an arbitrary idempotent $e\in A$ lies  in one of the subalgebras $A_i.$ Hence $\mathbf{st}  (  eAe)  =  \mathbf{st}  (  eA_i e)  \in  \text{Spec}  (A_i). $  \end{proof}

\begin{proof}[Proof of Theorem $\ref{Th_II.2.1_new}$ \emph{(1)}] To start with we notice that  $\{  1,  2,  \ldots,  n  \}$ $  =$ $ \text{Spec}  (M_n(\mathbb{F})). $ Let $s$ be a Steinitz  number. In \cite{BezOl}, we showed that there exists a unital locally matrix algebra $A,$ $\text{dim}_{\mathbb{F}}\, A\leq \aleph_0,$ such that $\mathbf{st}(A)=s.$ Consider the algebra  $M_{\infty}(A)$ of infinite $\mathbb{N}  \times \mathbb{N}$--matrices, having finitely many nonzero entries. The algebra $M_n(A)$ of $n \times n $--matrices over $A$ is embedded in $M_{\infty}(A)$ as a north--west corner, $$ M_1(A) \ \subset \ M_2(A) \ \subset \cdots , \ \ \ \  M_{\infty}(A) \ = \ \bigcup_{n= 1}^{\infty} \ M_n(A). $$ In particular, it implies that  $M_{\infty}(A)$ is a locally matrix algebra. We will show that
\begin{equation}\label{Spec_M_inf}
\text{Spec}  \big(  M_{\infty}(A)  \big) \ = \ S   (\infty,s) .
\end{equation}
Indeed, by Lemma \ref{Lem_II.2.2_new}, $$\text{Spec}  \big(  M_{\infty}(A) \big) \ = \ \bigcup_{n = 1}^{\infty} \ \text{Spec}  \big(  M_n(A)  \big).$$ We have $ \mathbf{st}  (M_n(A))  = n  s.$ In \cite{Morita_BO}, we showed that  $$\text{Spec}  \big( \ M_{n}(A) \ \big) \ = \ \Big\{ \ \frac{a}{b}  \ n \ s \ \Big| \  b  \in  \Omega  (ns), \ a,  b  \in  \mathbb{N}; \  1  \leq  a  \leq  b  \Big\} .$$ This implies  $\text{Spec} (  M_{n}(A) )  \subseteq  S  (\infty, s).$ A Steinitz  number $ (a/ b) \cdot  s,$ $b \in  \Omega  (s)$ lies in $\text{Spec}  (  M_{n}(A) )$ provided that  $a/b \leq n.$ This completes the proof of  (\ref{Spec_M_inf}). In particular, $\text{Spec} (  M_{\infty}(\mathbb{F}) )  =  \mathbb{N}.$

Consider now a  saturated subset $S =  S  (r,s)$ or $S  =  S^{+} (r,s),$ $ 1  \leq  r  <  \infty,$ where  $s$ is an infinite Steinitz  number. Choose a sequence $b_1,$ $ b_2,$ $ \ldots  \in  \Omega  (s)$  such that  $ b_i$ divides  $ b_{i+1},$ $ i  \geq  1, $ and  $s$ is the least common multiple of  $b_i,$ $i \geq 1.$ There exists a unique  (up to isomorphism) unital countable--dimensional locally matrix algebra $A_{s/b_i}$ such that  $ \mathbf{st}   (A_{s/b_i})  =  s/ b_i . $ Let $A_i   =  M_{r_s(b_i)} (  A_{s/b_i} ). $ We have $$\mathbf{st}  (A_{s/b_i} ) \ = \ \mathbf{st}   \Big( \big(  M_{b_{i+1}/b_i}  \big(  A_{s/b_{i+1}}  \big)  \Big) \ = \ s/b_{i+1}.$$ Hence, by Glimm's Theorem, $A_{s/b_i}  \cong   M_{b_{i+1}/b_i} (  A_{s/b_{i+1}} ) $ and, therefore, $$A_{i} \ = \  M_{r_s(b_i)} \ \big( \ A_{s/b_{i}} \ \big) \ \cong \ M_{r_s(b_i)\, \cdot\, \frac{b_{i+1}}{b_i}} \ \big( \ A_{s/b_{i+1}} \ \big).$$ By the inequality  (\ref{Big_frac_1}), $r_s(b_i)  \cdot  (b_{i+1} / b_i ) \leq  r_s(b_{i+1}).$ Hence, the algebra $A_i$ is embeddable in the algebra $A_{i+1} $ as a north--west corner. Let $$ A \ = \ \bigcup_{i= 1}^{\infty} \  A_i.$$  We will show that  $\text{Spec}  (A)=S.$ Let $0 \neq e\in A$ be an idempotent. Then  $e\in A_i$ for some   $i\geq 1.$ In \cite{Morita_BO}, we showed that  $$\mathbf{st}  (eA_i e) \ = \ \frac{a}{b} \ \mathbf{st}  (A_i), $$ where $a,$ $b\in \mathbb{N};$ $a,$ $b$ are coprime natural numbers; $ b  \in  \Omega  (  \mathbf{st}  (A_i) ),$ $a  \leq  b.$ Futhermore, $$ \mathbf{st}  (A_i) \ = \ r_s(b_i) \ \frac{s}{b_i}, \ \ \ \ \mathbf{st}  (eA_i e) \ = \ \frac{a}{b} \ r_s(b_i) \ \frac{s}{b_i}. $$ Let $d   =  \text{gcd}  (b, r_s(b_i)),$ $b  =  d b\,',$ $ r_s(b_i)  =  d  \cdot  r_s(b_i)' .$ So, $$\mathbf{st} ( eA_{i}e  ) \ = \ \frac{a\cdot r_{s}(b_{i})'}{b\,'} \ \cdot \ \frac{s}{b_{i}} \ \in \ \mathbb{SN}.$$ This implies that  $b\,'  \in   \Omega   ( s/ b_{i}).$ Therefore, $b\,'b_{i}\in \Omega  (s).$ To show that  $\mathbf{st}   ( eA_{i}e  ) $ lies in $S  (r,s) $ (respectively $S^{+}  (r,s)$) we need to verify that  $a\cdot r_{s}(b_{i})\,'  \leq   r\,b\,'\,b_{i}$ (respectively $a\cdot r_{s}(b_{i})\,'  <  r\,b\,'\,b_{i}$). Multiplying both sides of the inequality by $d$ we get $a\cdot r_{s}(b_{i}) \leq  r\,b\,b_{i}$ (respectively $a\cdot r_{s}(b_{i}) <  r\,b\,b_{i}$). This inequality holds since $a  \leq  b $ and $r_{s}(b_{i}) \leq r\cdot b_{i}$ (respectively $ a \leq  b$ and $ r_{s}(b_{i})  <  r\cdot b_{i} $). We proved that $\text{Spec} (A)\subseteq S.$

Let us show that $S\subseteq \text{Spec} (A).$ Consider a Steinitz  number $(a /b) \cdot  s  \in  S,$ where $ a, $ $ b  \in  \mathbb{N}; $ $b   \in  \Omega  (s),$ $ a \leq   r\,b $ in the case $S  =  S   (r,s) $ or $a  <  r\,b$ in the case $ S   =  S^{+}  (r,s).$

There exists a member of our sequence  $b_{i}$ such that  $b $ divides $b_{i},$ $b_{i}  =  k\cdot b,$ $ k  \in  \mathbb{N}.$ Then $(a/b) \cdot s  = (a\,k / b_{i}) \cdot s.$

We will show that $ a\,k\leq r_{s}(b_{i}).$ Indeed, multiplying both sides of the inequality by  $b$ we get $ab_{i} \leq  r_{s}(b_{i})  b.$ Let $S=S  (r,s).$ Then $a\leq rb.$
Since $a\in \mathbb{N}$ it implies  $a\leq [r\,b].$ Furthermore, $r_{s}(b_{i}) =  [r b_{i}]  =  [r b  k].$ So, it is sufficient to show that  $[r  b]  k  \leq  [r  b k].$ This inequality holds since  $[r b] k $ is an integer and $ [r b]  k  \leq  r  b k.$

Now suppose that  $S=S^{+}  (r,s).$ Then $a< rb,$
$$r_{s}(b_{i})=\left\{
                 \begin{array}{ll}
                  [r\ b_{i}], & \hbox{if} \  \ \ \ r \ b_{i} \ \not\in \ \mathbb{N}, \\
                  r\ b_{i}\ - \ 1, & \hbox{if} \ \ \ \   r \ b_{i} \ \in \ \mathbb{N}.
                 \end{array}
               \right.$$
There are three possibilities:

\begin{enumerate}
    \item[1)] $r\,b\in \mathbb{N}$ and, therefore,  $r\,b_{i}\in \mathbb{N}.$ In this case $a  \leq  r b  - 1,$ $ r_{s}(b_{i}) = r b_{i}  - 1.$ We have $a b_{i} \leq  (r  b  -  1)  b_{i}  \leq (r b_{i} -\ 1) b =  r_{s}(b_{i})  b;$
    \item[2)] $r\,b \not \in \mathbb{N},$ but $r\, b_{i}\in \mathbb{N}.$ In this case $a  \leq  [r  b],$ $r_{s}(b_{i}) =  r  b_{i}  -  1,$ we have
$a b_{i} \leq  [r b]  b_{i},$ $ r_{s}(b_{i})  b  =  (r b_{i}  -  1)  b.$ Hence, it is  sufficient to show that $[r  b]  k  \leq  r b_{i}  -  1  =  r  b  k  -  1.$
    The number  $[r\, b] \, k$ is an integer and $[r b]  k  <  r b k$ since $ [r b]  <  r b.$ This implies the claimed inequality;
    \item[3)] $r\, b_{i}\not \in \mathbb{N}$ and, therefore,  $r\, b\not \in \mathbb{N}.$ In this case $a  b_{i}  \leq  [r  b]  b  k,$ $r_{s}(b_{i}) b = [r  b k]  b$ and it remains to notice that $[r\, b]\,k\leq [r\,b\,k].$
\end{enumerate}
We showed that both for $S=S  (r,s)$ and for  $S=S^{+}  (r,s)$ there holds the inequality $ak \leq r_s(b_i).$

Recall that \ $A_{i} =  M_{r_{s}(b_{i})} ( A_{s/b_{i}} ).$ \ Consider the  north--east corner \  $M_{ak}\left ( A_{s/b_{i}} \right )$ of the algebra $A_{i}.$ We have $$\mathbf{st}  \left ( M_{ak}   \left( A_{s/b_{i}} \right)    \right)   \ =\ a \ k \ \cdot \ \frac{s}{b_{i}} \ = \ \frac{a}{b} \ s,$$
and, therefore,  $S\subseteq \text{Spec}  (A).$ This completes the proof of Theorem~\ref{Th_II.2.1_new}~$(1).$ \end{proof}

For the proof of Theorem \ref{Th_II.2.1_new} $(2)$ we will need several lemmas on extension of isomorphisms.

\begin{lemma}\label{Lem_II.2.3_new}  Let $A$ be a locally matrix algebra and  let $A_1$ be a subalgebra of $A$ such that $A_1 \cong M_{n}(\mathbb{F}).$ Then every automorphism of the algebra $A_1$ extends to an automorphism of the algebra   $A.$
\end{lemma}

\begin{proof} Let $e$ be the identity element of the subalgebra  $A_1.$ Then the corner  $eAe$  is a unital locally matrix algebra. Let  $C$ be  the centraliser of the subalgebra $A_1$ in $eAe.$
By Wedderbern's Theorem (see, \cite{Drozd_Kirichenko,Jacobson_3}) we have  $e\,A\,e  =A_1 \otimes_{\mathbb{F}}  C.$ An arbitrary automorphism  $\varphi$ of the subalgebra $A_1$  is inner, that is, there exists an invertible element $x$ of the subalgebra $A_1$  such that  $\varphi(a)  =  x^{-1}ax $ for all elements $ a \in A_1.$ Conjugation by the element $x\otimes e$ extends  $\varphi$ to an automorphism of the algebra $eAe.$ Consider the Peirce decomposition $$A \ = \ eAe \ + \ eA(1-e) \ +\ (1-e)Ae \ + \ (1-e)A(1-e),$$ and the mapping $$\tilde{\varphi} \colon \ A \ \ni \ a \ \mapsto \ x^{-1}ax \ + \ x^{-1}a(1-e) \ + \ (1-e)ax \ + \ (1-e)a(1-e).$$ The mapping $\tilde{\varphi}$ extends $\varphi$ and  $\tilde{\varphi}\in \text{Aut}(A).$ This completes the proof of the lemma. \end{proof}

\begin{lemma}\label{Lem_II.2.4_new} Let $A$ be a unital locally matrix algebra with an idempotent $e \neq 0.$ Then an arbitrary automorphism of the corner  $eAe$ extends to an automorphism of the algebra $A.$
\end{lemma}
\begin{proof} Suppose at first that an automorphism  $\varphi$ of the algebra $eAe$ is inner, and there exists an element  $x_e \in eAe$ that is invertible  in the algebra  $eAe$ such that  $ \varphi(a) = x_e^{-1} a x_e $ for an arbitrary element  $a  \in  e\,A\,e.$ The element $x=x_e + (1-e)$ is invertible in the algebra $A.$ So, conjugation by the element $x$ extends $\varphi.$

Now let $\varphi$ be an arbitrary automorphism of the corner $eAe.$ Let $A_1\subseteq A$ be a subalgebra such that $ 1,$ $ e  \in  A_1$ and $A_1  \cong  M_m(\mathbb{F})$ for some $m  \geq  1.$ Consider  $A_2 \subseteq A$ such that  $ A_1  \subseteq   A_2,$ $\varphi(e\,A_1\, e)  \subseteq  A_2$ and $ A_2   \cong  M_n(\mathbb{F})$ for some $n\geq 1.$ Consider the embedding $$\varphi : \ e\,A_1\, e \ \rightarrow \ \varphi \ (e \,A_1\, e) \ \subseteq \ e\,A_2\, e$$ that preserves the identity element $e.$ By Skolem--Noether Theorem (see, \cite{Drozd_Kirichenko}) there exists an invertible element $x_e \in eA_2 e$ such that   $ \varphi(a)  =  x_e^{-1}  a  x_e$ for an arbitrary element $ a  \in e\,A_1\, e. $

As noticed above, there exists an automorphism $\psi $ of the algebra $A$ that extends the automorphism $ eAe  \rightarrow  eAe,$ $a \mapsto  x_e^{-1}  a  x_e. $ The composition $\psi^{-1}\circ \varphi$ leaves all elements of the algebra  $eA_1 e$ fixed. Since it is sufficient to prove that the  automorphism $ \psi^{-1}\circ \varphi  \in  \text{Aut}  (eAe)$ extends to an  automorphism of $A$ we will assume without loss of generality that the  automorphism  $\varphi \in \text{Aut}  (eAe)$ fixes all elements of $eA_1 e.$

Let $C$ be the centraliser of the subalgebra $A_1$ in $A.$ Then $A =  A_1  \otimes_{\mathbb{F}}  C $ and $e Ae  =  e A_1 e  \otimes_{\mathbb{F}}  C.$ Since the subalgebra $e  \otimes_{\mathbb{F}}  C $ is the centraliser of  $e  A_1   e  \otimes_{\mathbb{F}}  C $ in the algebra $eAe$ it follows that $e  \otimes_{\mathbb{F}}  C $ is invariant with respect to the automorphism $\varphi.$ Hence, there exists an automorphism $\theta \in \text{Aut} (C)$ such that $ \varphi  (a  \otimes  c)  =   a  \otimes  \theta(c) $ for all elements $ a  \in  e A  e,$ $c  \in  C. $ So, the automorphism $ \tilde{\varphi}  (a  \otimes  c)  =  a  \otimes  \theta(c), $ $ a   \in  A_1,$ $ c  \in  C,$ extends $\varphi.$ This completes the proof of the lemma.  \end{proof}

\begin{lemma}\label{Lem_II.2.5_new}
Let $A$ be a unital locally matrix algebra with nonzero idempotents $e_{1},$ $e_{2}.$ An arbitrary isomorphism $\varphi \colon  e_{1}  A  e_{1}  \rightarrow  e_{2}  A  e_{2}$ extends to an automorphism of the algebra  $A.$
\end{lemma}
\begin{proof} There exists a subalgebra $A_1\subseteq A$ such that $1,$ $e_{1},$ $e_{2}\in A_1$ and $A_1 \cong M_n(\mathbb{F})$ for some $n\geq 1.$ Let $r_i$ be the matrix range of the idempotent $e_i$ in  $A_1, $ $i=1,2.$ In \cite{Morita_BO}, it was shown that $$\mathbf{st} \left ( e_{1}\ A \ e_{1} \right )\ = \ \frac{r_{1}}{n} \ \cdot \ \mathbf{st} (A),\quad \mathbf{st} \left( e_{2}\ A\ e_{2} \right)\ =\ \frac{r_{2}}{n} \ \cdot \ \mathbf{st}  (A).$$ Since  $e_{1} Ae_{1}  \cong   e_{2}  A  e_{2}$ it follows that $ r_{1} =  r_{2}.$ In the  matrix algebra $M_n(\mathbb{F})$ any two idempotents of the same range are conjugate via an automorphism. Hence, the idempotents $e_1,$ $e_2$ are  conjugate via an automorphism of $A_1.$ By Lemma \ref{Lem_II.2.3_new}, an arbitrary automorphism of $A_1$ extends to an automorphism of the algebra $A.$ Now the assertion of the lemma follows from Lemma \ref{Lem_II.2.4_new}. \end{proof}

\begin{lemma}\label{Lem_II.2.6_new} Let $A,$ $B$ be isomorphic unital locally matrix algebras with nonzero idempotents $e\in A,$  $f \in B.$   An arbitrary isomorphism $e  A  e  \rightarrow   f  B  f$ extends to an  isomorphism $ A  \rightarrow  B.$
\end{lemma}
\begin{proof} Let $\varphi: A \rightarrow  B ,$ $ \psi : e A e \rightarrow    f  B  f$ be isomorphisms. Then $$\varphi^{-1}\circ \psi \ \colon \  e \ A \ e \ \rightarrow \ \varphi^{-1}(f) \ A \ \varphi^{-1}(f)$$ is an isomorphism of two corners  of the algebra $A.$ By Lemma \ref{Lem_II.2.5_new}, $\varphi^{-1}\circ \psi$ extends to an automorphism $\chi$ of the algebra  $A.$ , the isomorphism  $\varphi\circ\chi$ extends $\psi.$ \end{proof}

\begin{lemma}\label{Lem_II.2.7_new} Let $A$ be a unital locally matrix algebra and let $s_{1},$ $s_{2}$ be Steinitz  numbers from $\emph{Spec} (A).$ Suppose that  $s_{2}/s_{1}>1.$ Let $e_{1}\in A$ be an idempotent such that $\mathbf{st}  (e_{1}Ae_{1})=s_{1}.$ Then there exists an idempotent $e_{2}>e_{1}$ such that  $\mathbf{st}  ( e_{2}Ae_{2} )=s_{2}.$
\end{lemma}
\begin{proof} Since $s_{2}\in \text{Spec}  (A)$ there exists an idempotent $e\,'\in A$ such that  $\mathbf{st}  (e\,'Ae\,') = s_{2}.$ Choose a subalgebra $A_1\subseteq A$ such that  $e_{1},$ $e\,'\in A_1$ and $A_1 \cong M_{n}(\mathbb{F}).$

Let $r_1,$ $r_2$ be the matrix ranges of $e_1,$ $e\,'$ in $ M_{n}(\mathbb{F}),$ respectively. In \cite{Morita_BO}, it was  shown that  $$\mathbf{st}   ( e_{1} \ A\ e_{1} ) \ = \ s_{1} \ = \ \frac{r_{1}}{n} \ \mathbf{st} (A), \quad \mathbf{st}  ( e\,' \ A \ e\,'  ) \ = \ s_{2} \ = \ \frac{r_{2}}{n} \ \mathbf{st}  (A).$$ Hence $r_{2}>r_{1}.$ Since every idempotent in the algebra $ M_{n}(\mathbb{F})$ is diagonalizable there exist automorphisms $\varphi,$ $\psi$ of the algebra $A_1$ such that $\psi(e\,')>\varphi (e_1).$  By Lemma \ref{Lem_II.2.3_new}, the automorphisms $\varphi,$ $\psi$ extend to automorphisms $\widetilde{\varphi},$ $\widetilde{\psi}$ of the algebra $A,$ respectively.

Let $ e_2=\widetilde{\varphi}\,^{-1}\left ( \psi\left( e\,'\right) \right ).$ Then $ e_{2}> e_1$ and $ \mathbf{st} ( e_{2} A e_{2} )  =  s_{2},$ which completes the proof of the lemma. \end{proof}

\begin{proof}[Proof of Theorem $\ref{Th_II.2.1_new}$ $(2)$] Let $A,$ $B$ be countable--dimensional locally matrix algebras, $\text{Spec} (A)=\text{Spec} (B).$
Choose bases $a_{1},$ $a_{2},$ $\ldots$ and $b_{1},$ $ b_{2},$ $\ldots$ in the algebras $A,$ $B,$ respectively.

We will construct ascending chains of corners $\{0\}  = A_{0}  \subset   A_{1}  \subset   A_{2}  \subset   \cdots$ in the algebra $A$ and $\{0\}  =  B_{0}  \subset  B_{1}  \subset  B_{2} \subset  \cdots $ in the algebra $B,$ such that  $$\bigcup_{i=0}^{\infty} \ A_{i} \ = \ A, \ \ \ \ \bigcup_{i=0}^{\infty} \ B_{i} \ = \ B$$ and $ a_{1}, $ $ \ldots,$ $ a_{i}  \in   A_{i},$ $ b_{1},$ $ \ldots,$ $ b_{i}  \in   B_{i},$ $\mathbf{st}  ( A_{i} )  =  \mathbf{st}  ( B_{i}  )$ for all $i \geq 1.$

Suppose that corners $\{0\}  = A_{0}  \subset   A_{1}  \subset   A_{2}  \subset   \cdots  \subset  A_n,$ $\{0\}  =  B_{0}  \subset B_{1}  \subset  B_{2}  \subset  \cdots  \subset  B_n$ have already been selected, $n \geq 0.$ There exist corners $A\,' \subset A,$ $B\,' \subset B$ in the  algebras  $A,$ $B,$ respectively, such that $A_n \subset  A\,',$ $a_{n+1}  \in  A\,'$ and $ B_n  \subset  B\,',$ $ b_{n+1}  \in  B\,' .$ The  Steinitz  numbers $\mathbf{st}  (A\,'),$ $\mathbf{st}  (B\,')$ lie in the same saturated subset of  $\mathbb{SN},$ therefore, they are rationally connected.

  Suppose that $\mathbf{st} (B\,')  \geq   \mathbf{st}  (A\,').$ Let $=e\,' $ be an idempotent of the algebra $A$ such that $A\,'=e\,'A e\,'.$ The Steinitz  number $\mathbf{st} (B\,')$ lies in  $\text{Spec} (A).$ Hence, by Lemma \ref{Lem_II.2.7_new}, there exists an idempotent $e\in A$ such that  $e  \geq   e\,' $ and $\mathbf{st}  (e A  e)  =  \mathbf{st}  (B\,'). $ Choose $ A_{n+1}  =  e  A   e, $ $ B_{n+1}  =  B\,'.$ The chains $\{0\}  = A_{0}  \subset   A_{1}  \subset   A_{2}  \subset   \cdots$  and $\{0\}  =  B_{0}  \subset  B_{1}  \subset  B_{2} \subset  \cdots $ have been constructed.

By Lemma \ref{Lem_II.2.6_new}, every isomorphism $A_i \rightarrow B_i$ extends to an isomorphism $A_{i+1} \rightarrow B_{i+1}.$ This gives rise to a sequence of isomorphisms  $\varphi_i : A_i   \rightarrow   B_i, $ $ i  \geq  0,$ where each  $\varphi_{i+1}$ extends  $\varphi_i.$ Taking the union  $\cup_{i\geq 0}  \varphi_i $ we get an isomorphism from the algebra $A$ to the algebra $B.$  This completes the proof of theorem. \end{proof}

\begin{proof}[Proof of Theorem $\ref{Th4_Spec}$] It is easy to see that a locally matrix algebra $A$ is a unital if and only if the set of idempotents of $A$ has a largest element:  an identity. This is equivalent to $\text{Spec}(A)$ containing a  largest Steinitz  number. Among saturated sets of Steinitz  numbers  only $\{1,2,\ldots,n\}$ and $S(r,s),$ $s\in \mathbb{SN} \diagdown \mathbb{N},$ $r=u/v;$ $u$ and $v$ are coprime natural numbers, $v\in \Omega(s),$ satisfy this assumption. \end{proof}

\section{Embeddings of locally matrix algebras}

\begin{lemma}\label{Lem_II.3.1_new} Let $S_1,$  $S_2$ be saturated sets of Steinitz  numbers. Then either $S_1 \bigcap S_2 = \emptyset$ or one of the sets $S_1,$ $S_2$ contains the other one. \end{lemma}
\begin{proof} Let $s\in S_1 \bigcap S_2.$ If $s\in \mathbb{N}$ then, by Lemma \ref{Lem_II.1.10_new}, each set $S_i$ is either a segment $[1,n],$ $n\geq 1,$ or the whole $\mathbb{N}.$ In this case  the assertion of the lemma is obvious.

Suppose that the number $s$ is infinite. Then by Theorem \ref{Th_II.1.11_new}, $S_i  =  S  (r_i,s) $ or $S_i  =  S^{+}  (r_i,s), $ where $  r_i  =  r_{S_i}(s)  \in  [1,  \infty)   \cup  \{\infty \}, $ $ i = 1,  2.$ Clearly, if $ r_{S_1}(s)  <   r_{S_2}(s)$ then $  S_1  \subsetneqq  S_2.$ If $r_{S_1}(s)  =  r_{S_2}(s) $ then $$ S_1, \ S_2 \ = \
\left[ \begin{array}{lc}
S  (r,s)    \\
S^{+}  (r,s)  \\
S  (\infty,s)
\end{array}
\right.$$
and  $S^{+}  (r,s)  \subseteq  S  (r,s)  \subset   S  (\infty,s)$ for any $ r  \in  [1,  \infty).$ This completes the proof of the lemma. \end{proof}

Let  $A$ be a locally matrix algebra. A subalgebra  $B \subseteq A$ is called an \emph{approximative corner} of $A$ if $B$ is the union of an increasing chain of corners. In other words, there exist idempotents  $e_0 ,$ $ e_1,$ $e_2,$ $\ldots$ such that $$e_0 \ A  \ e_0\  \subseteq \ e_1 \ A \ e_1 \ \subseteq  \ e_2 \ A  \ e_2 \ \subseteq \ \cdots, \ \ \ \  B \ = \ \bigcup_{i=0}^{\infty}  \ e_i \ A \ e_i.$$

It is easy to see that an approximative corner  of a locally matrix algebra is a locally matrix algebra.

\begin{theorem}\label{Lem_II.3.2_new} Let  $A,$ $B $ be countable--dimensional locally matrix algebras. Then $B$ is embeddable in $A$ as an approximative corner if and only if
 $\emph{Spec}  (B)$ $ \subseteq $ $ \emph{Spec}  (A).$
\end{theorem}
\begin{proof} If $B$ is an approximative corner of $A$ then  every corner of $B$ is a corner of $A$,  hence $\text{Spec}  (B)  \subseteq \text{Spec}  (A). $

Suppose now that $\text{Spec} (B)  \subseteq \text{Spec}  (A). $ If the algebra $B$ is unital then it embedds in the algebra $A$ as a corner. Indeed, the embedding $\text{Spec}  (B)  \subseteq \text{Spec}  (A) $ implies that there exists an idempotent  $e\in A$ such that  $ \mathbf{st}  (B) =   \mathbf{st}   (e A  e).$ By Glimm's Theorem \cite{Glimm}, we have $B\cong e A e.$

Suppose now that the algebra  $B$ is not unital. Then there exists a sequence of idempotents  $0  =  f_0,$ $f_1,$ $f_2,$ $ \ldots$ of algebra $ B$ such that   $$ \{0\} = \ f_0 \ B \ f_0  \ \subsetneqq  \ f_1 \ B \ f_1 \  \subsetneqq \  f_2 \ B \ f_2 \  \subsetneqq \  \cdots, \quad  \bigcup_{i= 0}^{\infty} \ f_i \ B \ f_i \ = \ B.$$ We will construct a sequence of idempotents  $ e_0,$ $ e_1, $ $ e_2,$ $\ldots$ in the algebra $A$ such that  $$ e_0 \ A \ e_0 \ \subsetneqq \ e_1 \ A \ e_1 \ \subsetneqq \ e_2 \ A \ e_2 \ \subsetneqq \ \cdots, \quad \mathbf{st}  (f_i \ B \ f_i) \ = \ \mathbf{st}  (e_i \ A \ e_i)$$ for an arbitrary $ i \geq  0.$ Let $e_0 =0.$ Suppose that we have already selected idempotents  $e_0,$ $e_1,$ $\ldots,$ $e_n\in A$ such that  $ e_0  A  e_0  \subset  e_1  A  e_1  \subset  \cdots  \subset  e_n  A  e_n $ and  $\mathbf{st} (e_i  A  e_i)  =  \mathbf{st}  (f_i  B  f_i),$ $ 0  \leq   i  \leq n.$ We have $$ \mathbf{st}  (f_{n+1} \ B \ f_{n+1}) \ > \ \mathbf{st}  (f_n \ B \ f_n) \ = \ \mathbf{st}  (e_n \ A \ e_n) $$ and $\mathbf{st}  (f_{n+1} B f_{n+1}) \in \text{Spec}(A).$  By Lemma \ref{Lem_II.2.7_new}, there exists an idempotent $e_{n+1}\in A$ such that  $ e_n  A  e_n  \subset  e_{n+1}  A    e_{n+1} $ and  $ \mathbf{st}  (e_{n+1} A  e_{n+1})  =  \mathbf{st}  (f_{n+1}  B  f_{n+1}),$ which  proves existence of a sequence $e_0,$ $e_1,$ $e_2,$ $\ldots.$

The union $$ A\,' \ = \ \bigcup_{i=0}^{\infty} \ e_i \ A \ e_i$$ is an approximative corner of the algebra $A.$ By Glimm's Theorem \cite{Glimm}, $e_i  A  e_i  \cong  f_i  B  f_i,$ $ i  \geq  1.$ By Lemma \ref{Lem_II.2.2_new},  $$ \text{Spec}  (A\,') \ = \ \bigcup_{i = 1}^{\infty} \ \text{Spec}  (e_i \ A \ e_i ) \quad \text{and} \quad \text{Spec}  (B) \  = \  \bigcup_{i= 1}^{\infty} \ \text{Spec} ( f_i  B  f_i).$$ Hence $ \text{Spec} (B) = \text{Spec} (A\,'). $ By Theorem $\ref{Th_II.2.1_new}$ $(2),$ we have  $A\,' \cong B,$ which completes the proof of the theorem. \end{proof}


\begin{thebibliography}{99}



\bibitem{Baranov2} A.~A.~Baranov,  Classification of the direct limits of involution simple associative algebras and the corresponding dimension groups, \emph{Journal of Algebra} \textbf{381} (2013) 73--95.



\bibitem{BezOl} O.~Bezushchak and  B.~Oliynyk, Unital locally matrix algebras  and Steinitz numbers, \emph{J. Algebra Appl.} \textbf{19}(9) (2020). Doi:10.1142/SO219498820501807.

\bibitem{BezOl_2} O.~Bezushchak and B.~Oliynyk, Primary decompositions of unital locally matrix algebras,  \emph{Bull. Math. Sci.} \textbf{10}(1) (2020). Doi:10.1142/S166436072050006X.



\bibitem{Morita_BO} O.~Bezushchak and B.~Oliynyk, Morita equivalent unital locally matrix algebras,  \emph{Algebra Discrete Math.} \textbf{29}(2) (2020) 173--179.

\bibitem{Diskme} J.~Dixmier, On some $C^{*}$-algebras considered by Glimm,  \emph{J. Funct. Anal.} \textbf{1} (1967) 182--203.

\bibitem{Drozd_Kirichenko} Yu.~A.~Drozd, V.~V.~Kirichenko, Finite dimensional algebras, \emph{Springer-Verlag, Berlin--Heidelberg--New York} (1994).


\bibitem{Glimm} J.~G.~Glimm,   On a certain class of operator algebras, {\em Trans. Amer. Math. Soc.} {\bf 95}(2) (1960) 318--340.



\bibitem{Jacobson_3} N.~Jacobson, \emph{Structure of rings}, Colloquium Publications, \textbf{37} (1956).

\bibitem{Kurochkin} V.~M.~Kurochkin. On the theory of locally simple and locally normal algebras, {\em Mat. Sb., Nov. Ser.} {\bf 22(64)}(3) (1948)  443--454.
\bibitem{Kurosh} A.~Kurosh, Direct decompositions of simple rings, 	{\em Rec. Math. [Mat. Sbornik] N.S.}  {\bf 11} ({\bf 53}) (3) (1942) 245--264.


\bibitem{ST} E.~Steinitz,  Algebraische Theorie der K{\"o}rper, {\em Journal f{\"u}r die reine und angewandte Mathematik} {\bf 137} (1910) 167--309.








\end{thebibliography}
\end{document}